\documentclass[a4paper,reqno]{amsart}
\usepackage{graphicx}
\usepackage{amsmath,amssymb,amsthm}
\usepackage{mathrsfs,subfig}
\usepackage{hyperref}
\usepackage{verbatim}
\usepackage{amsfonts}

\theoremstyle{definition}
\newtheorem{thm}{Theorem}[section]

\newtheorem{defi}[thm]{Definition}
\newtheorem{lemm}[thm]{Lemma}
\newtheorem{prop}[thm]{Proposition}
\newtheorem{cor}[thm]{Corollary}

\newtheorem{rem}[thm]{Remark}
\numberwithin{equation}{section}

\DeclareMathOperator{\Mod}{Mod }
\DeclareMathOperator{\im}{Im}
\DeclareMathOperator{\diam}{diam}
\DeclareMathOperator{\dist}{dist}
\DeclareMathOperator{\loc}{loc}
\newcommand{\mres}{\mathbin{\vrule height 1.6ex depth 0pt width
0.13ex\vrule height 0.13ex depth 0pt width 1.3ex}}
\DeclareMathOperator{\sgn}{sgn}

\begin{document}

\title{Quasiconformal mappings on the Grushin plane
}


\author{Chris Gartland}
\address{Department of Mathematics, University of Illinois at Urbana-Champaign, 1409 West Green St., Urbana, IL 61801} 
\email{cgartla2@illinois.edu}

\author{Derek Jung}
\email{jung2@illinois.edu}

\author{Matthew Romney} 
\email{romney2@illinois.edu} 

\date{8 April 2016}

\maketitle

\begin{abstract}
We prove that a self-homeomorphism of the Grushin plane is quasisymmetric if and only if it is metrically quasiconformal and if and only if it is geometrically quasiconformal. As the main step in our argument, we show that a quasisymmetric parametrization of the Grushin plane by the Euclidean plane must also be geometrically quasiconformal. We also discuss some aspects of the Euclidean theory of quasiconformal maps, such as absolute continuity on almost every compact curve, not satisfied in the Grushin case.
\end{abstract}

\section{Introduction}

For a given $\alpha \in \mathbb{N}$, the vector fields $X_1 = \frac{\partial}{\partial x_1}$ and $X_2 = |x_1|^\alpha \frac{\partial}{\partial x_2}$ determine a sub-Riemannian (Carnot-Carath\'eodory) metric on $\mathbb{R}^2$. This metric can be written more explicitly as  
\begin{equation} \label{equ:grushin_distance}
d_\alpha(x,y) = \inf_\gamma \int_0^1 \sqrt{x_1'(t)^2 + \frac{x_2'(t)^2}{|x_1(t)|^{2\alpha}}}\,dt,
\end{equation}
where the infimum is taken over all curves $\gamma: [0,1] \rightarrow \mathbb{R}^2$ with $\gamma(0) = x$ and $\gamma(1) = y$ that are absolutely continuous with respect to the Euclidean metric. We denote this metric by $d_\alpha$, and we define $\mathbb{G}_\alpha^2$ (the {\it $\alpha$-Grushin plane}) to be $\mathbb{R}^2$ equipped with the metric $d_\alpha$. This definition can be extended to all $\alpha \geq 0$ by taking (\ref{equ:grushin_distance}) as the definition for $d_\alpha$. We further obtain a metric measure space by equipping $\mathbb{G}_\alpha^2$ with the Hausdorff 2-measure generated by $d_\alpha$, which we denote by $\mathcal{H}_\alpha^2$. The $\alpha$-Grushin plane is a Riemannian manifold except on the {\it singular line} $Y = \{0\}\times \mathbb{R}$. If $\alpha$ is an integer then it is also a sub-Riemannian manifold. Background on the Grushin plane may be found in Bell\"aiche \cite{Bellaiche} and Monti and Morbidelli \cite{MonMor1}.

While initially considered only in Euclidean space, the theory of quasiconformal maps has been extended to various metric space settings. Of particular mention is the work of Heinonen, Koskela, Shanmugalingam and Tyson \cite{HKST:01}, which shows the equivalence of the metric, geometric, and analytic definitions in the setting of Ahlfors regular metric measure spaces supporting a Poincar\'e inequality. Previously, quasiconformal maps on the Heisenberg group, a fundamental example of a sub-Riemannian manifold, had been studied by Kor\'anyi and Reimann \cite{KorRei:95}, who proved the equivalence of appropriate metric, analytic, and geometric definitions. Results for the more general class of equiregular sub-Riemannian manifolds were subsequently obtained by Margulis and Mostow \cite{MarMos:95}. 

An investigation of quasiconformal maps in the Grushin plane was initiated by Ackermann in \cite{Ack}. In this paper we complete her goal of showing that all definitions of quasiconformality, when appropriately understood, are equivalent.  Our focus is on the geometric definition of quasiconformality, which was not considered in her paper.
We emphasize that the equivalence of definitions in the case of the Grushin plane is not covered by the existing literature; if $\alpha \geq 1$, the measure $\mathcal{H}_\alpha^2$ is not Ahlfors regular or even locally finite on the singular line, and as a sub-Riemannian manifold the Grushin plane is not equiregular. On the other hand, the Grushin plane is nicely behaved in another respect: by a result of Meyerson \cite{Meyerson}, the Grushin plane $\mathbb{G}_\alpha^2$ is quasisymmetrically equivalent to the Euclidean plane. A particular quasisymmetry is given by $$\varphi(x_1,x_2) = ((1+\alpha)^{-1}|x_1|^\alpha x_1, x_2),$$
which we will call the \emph{canonical quasisymmetry}. The map used by Meyerson did not include the constant $(1+\alpha)^{-1}$, though we have added it so that $\varphi$ is a conformal map between Riemannian manifolds outside the singular line. 

The main step in our proof (Proposition \ref{prop:varphi_gqc}) is to show that the quasisymmetry $\varphi$ is also geometrically quasiconformal. We conclude from this that any quasisymmetric parametrization of $\mathbb{G}_\alpha^2$ by the Euclidean plane must also be geometrically quasiconformal. This aspect of our work connects it with a fairly large body of recent literature on the \textit{quasisymmetric} (or \textit{quasiconformal}) \textit{uniformization problem}. This asks under what conditions a metric space admits a quasisymmetric (or quasiconformal) parametrization by a well-known model space, and how different classes of parametrizations relate to each other. As a notable recent example, see Rajala \cite{Raj16} for a characterization of metric spaces with locally finite Hausdorff 2-measure admitting a geometrically quasiconformal parametrization by the Euclidean plane. A nearly-universal assumption in the literature is that the measure on the space in consideration be locally finite, which as mentioned is not satisfied by the Grushin plane $\mathbb{G}_\alpha^2$ for $\alpha \geq 1$.  




Our main result is the following set of equivalences, which combines our work with that of Ackermann \cite{Ack}. For a homeomorphism $f: \mathbb{G}_\alpha^2 \rightarrow \mathbb{G}_\alpha^2$, we call $Y \cup f^{-1}(Y)$ the \emph{singular set} of $f$ and its complement the \emph{Riemannian set} of $f$. 

\begin{thm} \label{thm:main_theorem}
Let $f: \mathbb{G}_\alpha^2 \rightarrow \mathbb{G}_\alpha^2$  be a homeomorphism. The following are equivalent.
\begin{itemize}
\item[(a)] $f$ is quasisymmetric.
\item[(b)] $f$ is metrically quasiconformal.
\item[(c)] $f$ is metrically quasiconformal on its Riemannian set. 
\item[(d)] $f$ is geometrically quasiconformal.
\item[(e)] $f$ is uniformly locally geometrically quasiconformal on its Riemannian set. That is, there exists a $K\geq 1$ such that, for any point $x$ in the Riemannian set of $f$, there exists a neighborhood $U_x$ such that $f|U_x: U_x \rightarrow \mathbb{G}_\alpha^2$ is geometrically quasiconformal with constant $K$. 
\item[(f)] $\varphi \circ f$ satisfies the absolute continuity on lines (ACL) property in its Riemannian set and satisfies the {\it conjugated Beltrami equation}
$$\left(\frac{\partial}{\partial x_1} + i |x_1|^\alpha \frac{\partial}{\partial x_2} \right)(\varphi \circ f)(x) = \mu(x) \left(\frac{\partial}{\partial x_1} - i |x_1|^\alpha \frac{\partial}{\partial x_2} \right)(\varphi \circ f)(x)$$
for a.e. $x = (x_1,x_2)$ in its Riemannian set, for some measurable function $\mu: \mathbb{G}_\alpha^2\setminus (Y \cup f^{-1}(Y)) \rightarrow \mathbb{C}$ with $\|\mu\|_\infty < 1$. Here we consider $\varphi \circ f$ as a function into $\mathbb{C}$.   
\end{itemize}
This result is quantitative in the sense that worst-case data for one definition can be computed from the data of any other definition. See Remark \ref{rem:quantitative} for details.
\end{thm}

Implicit in the statement of conditions (d) and (e) is a choice of exponent and a choice of measure: we consider here the exponent $p = 2$ and the measure $\mathcal{H}_\alpha^2$, the Hausdorff 2-measure generated by $d_\alpha$. Also, note that we use the so-called {\it ring definition} of geometric quasiconformality, which differs from other potential definitions. Precise definitions will be given in Section \ref{sec:definitions}.

The implication (a) $\Longrightarrow$ (b) $\Longrightarrow$ (c) is immediate, as is (d) $\Longrightarrow$ (e). The implication (a) $\Longleftrightarrow$ (f) was worked out in Ackermann \cite{Ack}. The implications (a) $\Longrightarrow$ (e) and (c) $\Longleftrightarrow$ (e) are immediate from the theory for Euclidean space and the fact that $\varphi$ is a conformal map between Riemannian manifolds outside the singular line. What remains to be shown are the two implications (c) $\Longrightarrow$ (a) and (a) $\Longrightarrow$ (d). The first of these is short and is given as Proposition \ref{thm:qs_equivalence}. The second of these is more difficult and will be worked out in Section \ref{sec:gqc}. It requires an auxiliary proposition (Proposition \ref{prop:hausdorff_intersection}) that a rectifiable curve cannot intersect the singular line in a set of positive Grushin Hausdorff 1-measure. This is used to prove that the quasisymmetry $\varphi$ is also geometrically conformal (Proposition \ref{prop:varphi_gqc}), after which the proof that (a) $\Longrightarrow$ (d) is immediate.  


As an additional remark, we point out that an analogue of Theorem \ref{thm:main_theorem} can be stated for proper subdomains in $\mathbb{G}_\alpha^2$ and $\mathbb{R}^2$, where local quasisymmetry takes the place of quasisymmetry, though for the sake of brevity we do not pursue this here. See Heinonen \cite[Theorem 11.12]{hei:lectures} or \cite[Theorem 9.8]{HKST:01} for an example of such a local result. Similarly, an analogous result for homeomorphisms $f: \mathbb{R}^2 \rightarrow \mathbb{G}_\alpha^2$ and $f: \mathbb{G}_\alpha^2 \rightarrow \mathbb{R}^2$ can be given. 
Furthermore, Theorem \ref{thm:main_theorem} can be established more easily in the case that $\alpha<1$, in which it is known that the $\alpha$-Grushin plane is bi-Lipschitz equivalent to $\mathbb{R}^2$ by work of Romney and Vellis \cite{RV} and Wu \cite{Wu:grushin}. 

The final sections of this paper contain a number of auxiliary results. In Section \ref{sec:conformal_maps} we give a description of conformal self-mappings of the Grushin plane. In Section \ref{sec:exm} we discuss an example of an interesting curve family to illustrate some differences with the Euclidean theory (Theorem \ref{thm:nonrectifiable_family} and Corollary \ref{cor:failure}).

\section{Definitions of quasiconformality} \label{sec:definitions}

We give a review of the three main definitions of quasiconformality, as well as the related notion of quasisymmetry.  First is geometric quasiconformality, which applies to any two metric measure spaces $(X, d_X, \mu)$ and $(W, d_W, \nu)$.  For a fixed $p \geq 1$, define the $p$-modulus of a family $\Gamma$ of curves in $X$ as
$$\Mod_p \Gamma := \inf \int_X \rho^p d\mu,$$
where the infimum is taken over all Borel measurable functions $\rho: X \rightarrow [0,\infty]$ with the property that $\int_\gamma \rho\, ds \geq 1$ for all locally rectifiable curves $\gamma \in \Gamma$. Such a function $\rho$ is called \emph{admissible}, and the value $p$ is called the \emph{exponent}. For disjoint continua $E, F \subset X$, we let $\Gamma(E,F)$ denote the family of curves with one endpoint in $E$ and another endpoint in $F$. Any curve in $\Gamma(E,F)$ is defined on a closed interval, and so local rectifiability coincides with rectifiability in this case.
\begin{defi} \label{defi:geometric_qc}
A homeomorphism $f: (X, d_X, \mu) \rightarrow (W, d_W, \nu)$ between metric measure spaces is \emph{geometrically quasiconformal with exponent $p$} if there exists a constant $K$ such that 
\begin{equation} \label{equ:geometric_qc}
K^{-1}\Mod_p \Gamma(E,F) \leq \Mod_p f\Gamma(E,F) \leq K \Mod_p \Gamma(E,F)
\end{equation} 
for all disjoint continua $E, F \subset X$. 
\end{defi}
In all cases in this paper we will take $p=2$ as the exponent, and we write $\Mod \Gamma$ in place of $\Mod_2 \Gamma$. We will always equip $\mathbb{G}_\alpha^2$ with the Grushin Hausdorff 2-measure $\mathcal{H}_\alpha^2$. Since the vector fields $X_1 = \frac{\partial}{\partial x_1}$ and $X_2 = |x_1|^\alpha \frac{\partial}{\partial x_2}$ form an orthonormal basis for the tangent space at each point on the Riemannian part of $\mathbb{G}_\alpha^2$, we obtain the relationship $d\mathcal{H}_\alpha^2 \mres (\mathbb{G}_\alpha^2 \setminus Y) = |x_1|^{-\alpha}dx_1dx_2$. On the other hand, using the dilation property (\ref{equ:grushin_dilation}) on the singular line, we have that $\mathcal{H}_\alpha^2\mres Y = C\mathcal{H}_E^{2/(1+\alpha)}$ for some constant $C=C(\alpha)$. For the Euclidean plane $\mathbb{R}^2$, we always use the Lebesgue 2-measure $dm$, which is equal to the Euclidean Hausdorff 2-measure $\mathcal{H}_E^2$ (with appropriate scaling). 

In the definition of geometric quasiconformality, it is common to require that (\ref{equ:geometric_qc}) hold for all curve families $\Gamma$, not just those curve families corresponding to some disjoint continua $E,F$. These two definitions are equivalent in the Euclidean case (with the same constant $K$; see \cite[Corollary 36.2]{Vais1}), though, as we will see, this is not the case for the Grushin plane. See Section \ref{sec:exm} for discussion on this matter.

Next, we consider together metric quasiconformality and quasisymmetry. For these definitions we no longer require that the metric spaces carry a measure. 

\begin{defi}
A homeomorphism $f: (X, d_X) \rightarrow (W, d_W)$ is \emph{metrically quasiconformal} if there exists $H \geq 1$ such that, for all $x \in X$,
$$\limsup_{r \rightarrow 0} \frac{\sup\{d_W(f(x),f(y)): y \in X, d_X(x,y) \leq r\}}{\inf\{d_W(f(x),f(y)): y \in X, d_X(x,y) \geq r\}} \leq H.$$ 
\end{defi}

\begin{defi}
A homeomorphism $f: (X, d_X) \rightarrow (W, d_W)$ is \emph{quasisymmetric} if there exists a homeomorphism $\eta: [0, \infty) \rightarrow [0, \infty)$ such that for all triples of distinct points $x,y,z \in X$,
$$\frac{d_W(f(x),f(y))}{d_W(f(x),f(z))} \leq \eta\left(\frac{d_X(x,y)}{d_X(x,z)} \right) .$$
\end{defi}
It is immediate from these definitions that any quasisymmetric map is also metrically quasiconformal with $H = \eta(1)$. Quasisymmetry is a strong condition requiring control on the relative distortion of triples of points at all scales; the definition of metric quasiconformality requires that a similar condition hold only infinitesimally. 


We conclude now with analytic quasiconformality. This definition as stated applies only to homeomorphisms of $\mathbb{R}^n$, although more general metric definitions exist. See for example \cite[Section 9]{HKST:01}. 

\begin{defi}
A homeomorphism $f: \mathbb{R}^n \rightarrow \mathbb{R}^n$ (where $n \geq 2$) is \emph{analytically quasiconformal} if $f$ is in the Sobolev space $W_{\loc}^{1,n}(\mathbb{R}^n, \mathbb{R}^n)$, 
and there exists $K \geq 1$ such that
\begin{equation} \label{equ:analytic_qc} 
|Df(x)|^n \leq KJ_f(x)
\end{equation}
for a.e. $x \in \mathbb{R}^n$. Here $J_f$ is the Jacobian of $f$, and $|Df|$ is the operator norm of $Df$. 
\end{defi}

If $n=2$, after identifying $\mathbb{R}^2$ with $\mathbb{C}$, the condition (\ref{equ:analytic_qc}) is equivalent to satisfying the Beltrami equation: there exists  a measurable function $\mu: \mathbb{C} \rightarrow \mathbb{C}$, $\|\mu\|_\infty \leq (K-1)/(K+1) < 1$, such that $\frac{\partial f}{\partial \overline{z}} = \mu \frac{\partial f}{\partial z}$ a.e. 
Also, the requirement that $f \in W_{\loc}^{1,n}(\mathbb{R}^n, \mathbb{R}^n)$ is equivalent to the so-called ACL$^n$ property (see for instance Rickman \cite[Theorem 1.2]{Ric:93}). The map $f$ is {\it absolutely continuous on lines} (ACL) if for each closed $n$-interval $Q = [a_1, b_1] \times \cdots \times [a_n,b_n]$, $f$ is absolutely continuous on $n$-a.e. line segment in $Q$ parallel to a coordinate axis (that is, the family of those line segments where $f$ fails to be absolutely continuous has $n$-modulus zero). Then $f$ is said to be ACL$^n$ if additionally the partial derivatives (which must exist almost everywhere) are locally $n$-integrable.

Informally, we will refer to the constants $H$ and $K$ and the function $\eta$ above as the \emph{data} corresponding to a given definition. 

In the classical setting of homeomorphisms of $\mathbb{R}^n$, all four definitions are equivalent (taking $p = n$ in the geometric definition). It is quantitative in the sense that worst-case data corresponding to one definition can be computed from the data of the other. This equivalence is a deep result that was worked out by several mathematicians---principally Ahlfors, Gehring, and V\"ais\"al\"a---over many years. A detailed overview of this is given in the book of V\"ais\"al\"a \cite{Vais1}.  

As mentioned above, the proof of Theorem \ref{thm:main_theorem} can be accomplished using the quasisymmetric map $\varphi$ together with the equivalence of definitions in the Euclidean case. We can prove right away the implication (c) $\Longrightarrow$ (a). 

\begin{prop}\label{thm:qs_equivalence}
If a homeomorphism $f: \mathbb{G}_\alpha^2 \rightarrow \mathbb{G}_\alpha^2$ is metrically quasiconformal on its Riemannian set, then it is quasisymmetric. 
\end{prop}
\begin{proof}
Let $\widetilde{f} = \varphi \circ f \circ \varphi^{-1}$, and let $\widetilde{Y}$ denote the vertical axis as a subset of $\mathbb{R}^2$. Notice that $\mathbb{R}^2 \setminus \widetilde{f}^{-1}(\widetilde{Y})$ has two components, which we denote by $E_1$ and $E_2$. Now, $\widetilde{f}$ is metrically quasiconformal on each component of $E_i \setminus \widetilde{Y}$. We apply a theorem of V\"ais\"al\"a \cite[Theorem 35.1]{Vais1}, which states that any set of $\sigma$-finite 1-dimensional Lebesgue measure is negligible, to conclude that $\widetilde{f}$ is metrically quasiconformal on both $E_1$ and $E_2$. It follows then that $\widetilde{f}^{-1}$ is metrically quasiconformal on $\mathbb{R}^2 \setminus \widetilde{Y}$. By the same theorem of V\"ais\"al\"a, $\widetilde{f}^{-1}$ is metrically quasiconformal on all of $\mathbb{R}^2$. This suffices to show that $\widetilde{f}$, and hence $f$ itself, is quasisymmetric. 
\end{proof}

As a particular consequence, it follows that the inverse of a metrically quasiconformal homeomorphism $f: \mathbb{G}_\alpha^2 \rightarrow \mathbb{G}_\alpha^2$ is also metrically quasiconformal. This answers a question of Ackermann \cite[p. 315]{Ack}. 

\section{Geometric quasiconformality}\label{sec:gqc}

This section is devoted to proving the implication (a) $\Longrightarrow$ (d) that every quasisymmetric map is geometrically quasiconformal, although we require some technical groundwork. In preparation for the following lemma, we recall the following dilation property of $\mathbb{G}_\alpha^2$:

\begin{equation} \label{equ:grushin_dilation}
d_\alpha((\lambda x_1,\lambda^{1+\alpha} y_1),(\lambda x_2,\lambda^{1+\alpha} y_2)) = \lambda d_\alpha((x_1,y_1),(x_2,y_2)) .
\end{equation}
In particular, for any points $(0,b_1), (0,b_2) \in Y$, $d_\alpha((0,b_1),(0,b_2)) = C|b_2 - b_1|^{1/(1+\alpha)}$, where $C = d_\alpha((0,0), (0,1))$. Hence the space $(Y, d_\alpha|_{Y\times Y})$ is isometric, up to rescaling of the metric, to the {\it $\beta$-snowflake line} $(\mathbb{R}, |\cdot|^\beta)$ for $\beta =1/(1+\alpha)$. 

For the statement of the following lemma, we say that a metric $d$ on a totally ordered space 
$Y$ is {\it monotone} if for all $w,x,y,z \in Y$ satisfying $w\le x \le y \le z$ it holds that $d(x,y) \le d(w,z)$. Notice that any $\beta$-snowflake line is monotone, and hence $(Y, d_\alpha|_{Y \times Y})$, is monotone when considered with the total ordering given by comparing their second coordinates. 

\begin{lemm} \label{lemm:totally_ordered} Let $(X,d)$ be a metric space and let $Y$ be a totally ordered subset of $X$ such that $d|_{Y\times Y}$ is monotone.
Given any points $y_0<y_1 < \cdots <y_n$ in $Y$ and a permutation 
$\sigma$ of $\{0, 1, \ldots , n\}$,
$$\sum_{k=1}^n d(y_{k-1},y_k) \leq \sum_{j=1}^n d(y_{\sigma(j-1)}, y _ {\sigma(j)}).$$
In particular, 
if $\gamma:I\to X$ is any rectifiable curve whose image contains points $y_0<y_1 < \cdots <y_n $ of $Y$, then
$\text{length}(\gamma)\ge \sum_{k=1}^n d(y_{k-1},y_k)$.
\end{lemm} 
\begin{proof}
We use induction on $n$. The cases $n=0$ and $n=1$ are clear. 
Now, fix $n \geq 1$ and assume the lemma holds for all $k\le n$. Fix points $y_0 < y_1 < \cdots < y_{n+1}$ in $Y$ and a permutation $\sigma$ of $\{0, 1, \ldots , n+1\}$.
Choose $a,b \in \{0, 1, \ldots , n+1\}$  such that $\sigma(a) = 0$ and $\sigma(b) = n+1$.
By replacing $\sigma$ by $\tilde{\sigma}$, $\tilde{\sigma}(k):= n+1-\sigma(k)$  if necessary, we may assume that $a<b$.

Let $\tau$ be the permutation of $\{0, 1, \ldots , n+1\}$ satisfying $\tau(j) = \sigma(j)$ for all $0 \le j\le a$ and
$\tau(a+1)< \tau(a+2) <  \cdots < \tau(n+1)=n+1$.  By the inductive assumption, $\sum_{j=a+2}^{n+1}  d(y_{\sigma(j-1)} , y_{\sigma(j)}) \ge \sum_{k=a+2}^{n+1} d(y_{\tau(k-1)}, 
y_{\tau(k)}).$
Since $y_0 < y_{\tau(a+1)} \le y_{\sigma(a+1)}$, it follows from monotonicity that $\sum_{j=1}^{n+1}  d(y_{\sigma(j-1)} , y_{\sigma(j)}) \ge \sum_{k=1}^{n+1} d(y_{\tau(k-1)}, 
y_{\tau(k)}).$

\sloppy Since $\tau(n+1) = n+1$, the inductive assumption implies that $\sum_{k=1}^n d(y_{\tau(k-1),\tau(k)} ) \ge \sum_{l=1}^n d(y_{l-1},y_l).$ Since $y_{\tau(n)} \le y_n <y_{n+1}$, monotonicity implies that $\sum_{k=1}^{n+1} d(y_{\tau(k-1)} , y_{\tau(k)} ) \ge \sum_{l=1}^{n+1} d(y_{l-1}, y_l)$. The lemma follows. 
\end{proof}

We use this lemma to prove the following key fact, which we state in somewhat more generality than required.

\begin{prop} \label{prop:hausdorff_intersection}
Let $(X,d_X)$ be a metric space and $Y \subset X$ a subspace similar to the snowflaked line $(\mathbb{R},d_E^\beta)$ for some $\beta \in (0,1)$. For any rectifiable curve $\gamma: I \to X$, it holds that  $\mathcal{H}_X^1(\im(\gamma) \cap Y) = 0$.
\end{prop}
\begin{proof}
For simplicity, we will identify $Y$ with $(\mathbb{R},d_E^\beta)$ after rescaling the metric $d_X$ if necessary. Let $\gamma: I \rightarrow X$ be a curve, and let $Y_\gamma = \im(\gamma) \cap Y$. For simplicity, we will assume that $Y_\gamma \subset [0,1]$ and $\{0,1\} \subset \im(\gamma)$. 
We begin with the observation that $\mathcal{H}^1_X(A) = \mathcal{H}_E^\beta(A)$ for any subset $A \subset Y$. Here $\mathcal{H}_E^\beta$ refers to the $\beta$-Hausdorff measure relative to the Euclidean metric. Since $\mathcal{H}^1_X(Y_\gamma) \leq \ell(\gamma) < \infty$, it follows that $\mathcal{H}_E^1(Y_\gamma) = 0$. Now $U = [0,1] \setminus Y_\gamma$ is open and thus the union of countably many disjoint open intervals $U_k = (a_k, b_k)$. 

For a given $k \in \mathbb{N}$, the set $[0,1] \setminus \bigcup_{i=1}^k U_i$ consists of $k+1$ disjoint closed intervals, which we will denote by $F_k^j$, $1 \leq j \leq k+1$. Notice that the collection $\{F_k^j\}_{j=1}^{k+1}$ covers $Y_\gamma$; hence it will suffice to show that $\lim_{k \rightarrow \infty} \sum_{j=1}^{k+1} \diam_X(F_k^j) = 0$.   

Now for any fixed $k$, $\sum_{j=1}^k |b_j - a_j|^\beta \leq \ell(\gamma)$; this follows from Lemma \ref{lemm:totally_ordered} and the fact that for all $j \in \mathbb{N}$ the points $a_j, b_j$ are in $Y_\gamma$. This implies in particular that $\sum_{j=1}^\infty |b_j - a_j|^\beta \leq \ell(\gamma) < \infty$.
For a fixed $k \in \mathbb{N}$, the open intervals $U_{k+1}, U_{k+2}, \ldots$ are each contained in one of the closed intervals $F_k^j$. For each $1 \leq j \leq k+1$, let $\mathscr{U}_k^j$ denote the subcollection of $\{U_i\}_{i=k+1}^\infty$ consisting of those sets contained in $F_k^j$. Since $\mathcal{H}_E^1(Y_\gamma) = 0$, it must hold that $\sum_{U_i \in \mathscr{U}_k^j} |b_i - a_i| = \diam_E(F_k^j) =  \diam_X(F_k^j)^{1/\beta}$. By the subadditivity of the map $x \mapsto x^\beta$ we obtain $\diam_X(F_k^j) \leq \sum_{U_i \in \mathscr{U}_k^j} |b_i - a_i|^\beta$ and hence
$$\sum_{j=1}^{k+1} \diam_X(F_k^j) \leq \sum_{i=k+1}^\infty |b_i - a_i|^\beta. $$ 
The right-hand side tends to zero as $k \rightarrow \infty$, which establishes the proposition. 
\end{proof}

\begin{rem}
While this remark is not needed for our main result, we point out that the previous proposition is sharp for the $\alpha$-Grushin plane. For any $s<1$, it is possible to construct a rectifiable curve $\gamma$ such that $\mathcal{H}_\alpha^s(\im(\gamma) \cap Y)>0$. This can be done using an appropriate Cantor set construction. Identify $\{0\} \times [0,1] \subset Y$ with the interval $[0,1]$. Choose $L \in (0, 1/2)$ to satisfy $s = (1+\alpha)\log 2/\log(1/L)$. Let $V_1^1 = [a_1^1, b_1^1]$ be an interval of (Euclidean) length $1-2L$ centered at $1/2$. Let $V_2^1 = [a_2^1, b_2^1]$ and $V_2^2 = [a_2^2, b_2^2]$ be intervals of length $L(1-2L)$ centered in each component of $[0,1] \setminus V_1^1$. Define inductively intervals $V_n^j$ of length $L^{n-1}(1-2L)$ in a similar manner. Observe that $[0,1]\setminus \bigcup_{n,j} V_n^j$ is a Cantor set of Euclidean Hausdorff dimension $\log2/\log(1/L)$, which we denote by $E$. The Grushin Hausdorff dimension of $E$ is then equal to $s$. 

For each $n,j$, define $\gamma|[a_n^j, b_n^j] \rightarrow \mathbb{G}_\alpha^2$ by traversing a geodesic connecting $(0,a_n^j)$ to $(0,b_n^j)$ at constant speed. Then define $\gamma: [0,1] \rightarrow \mathbb{G}_\alpha^2$ by completion. Notice that $E = \im(\gamma) \cap Y$, and that $\ell(\gamma|[a_n^j, b_n^j]) = (b_n^j - a_n^j)^{1/(1+\alpha)} = (L^{n-1}(1-2L))^{1/(1+\alpha)}$. Hence 
$$\ell_\alpha(\gamma) = \sum_{n=1}^\infty 2^{n-1}(L^{n-1}(1-2L))^{1/(1+\alpha)}.$$
By the ratio test, this converges when $L<1/2^{1+\alpha}$ and diverges when $L> 1/2^{1+\alpha}$. Observe that the first case occurs precisely when $s<1$. 

By concatenating curves $\gamma_n$  of length $1/2^n$ such that $\mathcal{H}_\alpha^{1-1/n}(\im(\gamma) \cap Y) > 0$, we obtain a curve with $\dim_{\mathcal{H}_\alpha}(\im \gamma \cap Y) = 1$. 
\end{rem} 

Our next goal is to prove Proposition \ref{prop:varphi_gqc} that the canonical quasisymmetry $\varphi$ is also geometrically conformal. We require another lemma.

\begin{lemm}\label{lemm:path_redefine} 
Let $\rho \in L^2(\mathbb{R}^2)$ be a nonnegative function. For all $\delta, \epsilon >0$, there exists a set $A \subset (0,\delta)$ of positive measure such that $\int_{\gamma_r} \rho\, ds_E < \epsilon$ for all $r \in A$, where $\gamma_r$ is the curve traversing the circle of radius $r$ centered at the origin. 
\end{lemm}
\begin{proof} 
Fix $\epsilon > 0$ and let $R \in (0, \delta)$. We use $B_R$ to denote the ball of radius $R$ centered at the origin. Applying Fubini's theorem and H\"older's inequality gives
\begin{align*}
\int_0^R \left( \int_{\gamma_r} \rho\, ds_E \right)dr   = & \int_0^R \int_0^{2\pi} \rho(re^{i\theta}) r\,d\theta \,dr \\ & = \int_{B_R} \rho\,  dm 
\leq \sqrt{\pi}R \left( \int_{B_R} \rho^2\, dm \right)^{1/2}. 
\end{align*}
For sufficiently small $R$, we have $\int_{B_R} \rho^2\, dm < \epsilon^2/(2\pi)$ by the dominated convergence theorem, and so
$$\frac{1}{R} \int_0^R \left( \int_{\gamma_r} \rho\, ds_E \right)dr \leq \epsilon/\sqrt{2}.$$ The existence of the set $A$ follows.  
\end{proof}

Finally, before continuing with Proposition \ref{prop:varphi_gqc}, we write out a useful change of variables computation which holds for any $\rho \in L^2(\mathbb{G}_\alpha^2)$ with $\rho|_Y = 0$. This actually follows from the fact that $\varphi$ is a conformal map between Riemannian manifolds outside $Y$, but it seems beneficial to give full details.  We will use $(x,y)$ for coordinates in $\mathbb{G}_\alpha^2$ and $(u,v)$ for coordinates in $\mathbb{R}^2$. Notice that $(x,y) = \varphi^{-1}(u,v) = ((1+\alpha)^{1/(1+\alpha)} \sgn(u)|u|^{1/(1+\alpha)},v)$, though for convenience we will write simply $x=x(u)$. Also note that $|J \varphi| = |x|^\alpha$. 
Let $\gamma$ be a curve in $\mathbb{G}_\alpha^2$, let $\widetilde{\gamma} = \varphi \circ \gamma$, and let $\widetilde{\rho}(u,v) = |x(u)|^{-\alpha} \rho(x(u),v)\chi_{\mathbb{R}^2 \setminus \widetilde{Y}}$, where as before $\widetilde{Y}$ is the vertical axis of $\mathbb{R}^2$. We have 
\begin{align*}
\int_{\gamma} \rho\, ds_\alpha = \int_\gamma \rho(x,y) \sqrt{dx^2 + \frac{dy^2}{|x|^{2\alpha}}} = \int_{\widetilde{\gamma}} \frac{\rho(x(u),v)}{|x(u)|^\alpha}\sqrt{du^2+ dv^2} = \int_{\widetilde{\gamma}} \widetilde{\rho}\, ds_E
\end{align*}
and 
\begin{align*}
\int_{\mathbb{G}_\alpha^2} \rho^2\, d\mathcal{H}_\alpha^2 & = \iint_{\mathbb{G}_\alpha^2} \rho(x,y)^2\frac{dx\,dy}{|x|^\alpha} = \iint_{\mathbb{R}^2} \left(\frac{\rho(x(u),v)}{|x(u)|^\alpha}\right)^2 du\,dv \\ & = \int_{\mathbb{R}^2} \widetilde{\rho}^2 dm. 
\end{align*}

\begin{prop} \label{prop:varphi_gqc}
The map $\varphi: \mathbb{G}_\alpha^2 \rightarrow \mathbb{R}^2$ is geometrically conformal, that is, geometrically quasiconformal in the sense of Definition \ref{defi:geometric_qc} with $K=1$.
\end{prop}
\begin{proof}
Let $E, F \subset \mathbb{G}_\alpha^2$ be disjoint continua, and let $\Gamma$ be the family of curves connecting $E$ and $F$. Let $\widetilde{E}, \widetilde{F}$, and $\widetilde{\Gamma}$ denote their respective images under $\varphi$. Notice that if $\gamma \in \Gamma$ is rectifiable, then so is $\widetilde{\gamma} = \varphi \circ \gamma$. Recall as well that $\mathcal{H}_\alpha^1(\im \gamma \cap Y) = 0$ by Proposition \ref{prop:hausdorff_intersection}. Hence if $\widetilde{\rho}$ is admissible for $\widetilde{\Gamma}$, then $\rho(x,y) = \widetilde{\rho}(x|x|^\alpha,y)|x|^\alpha\chi_{\mathbb{G}_\alpha^2 \setminus Y}$ is admissible for $\Gamma$ with $\int_{\mathbb{G}_\alpha^2} \rho^2d\mathcal{H}_\alpha^2 = \int_{\mathbb{R}^2} \widetilde{\rho}^2 dm$ by the above change of variables argument. This establishes that $\Mod \Gamma \leq \Mod \widetilde{\Gamma}$. 

The reverse inequality $\Mod \widetilde{\Gamma} \leq \Mod \Gamma$ is more difficult. If $\widetilde{E}$ or $\widetilde{F}$ consists of a single point, then $\Mod \widetilde{\Gamma}=0$ and the claim follows trivially. Hence we may assume that neither $E$ nor $F$ is a single point. By the Riemann mapping theorem for doubly connected domains, there exists an annulus $A = \{z \in \mathbb{R}^2: 1 < |z| < r\}$ for some $r>1$ and a conformal map $\psi: A \rightarrow \widetilde{A}$, where $\widetilde{A}$ is the doubly-connected domain in the extended Euclidean plane $\widehat{\mathbb{R}}^2$ separating $\widetilde{E}$ and $\widetilde{F}$. Let $\Gamma_r$ be the family of radial curves connecting the two components of $\partial A$ and let $\widetilde{\Gamma}_r = \psi(\Gamma_r)$, where we remove from $\widetilde{\Gamma}_r$ the curve passing through the point at infinity should such exist. Observe that $\Mod \Gamma_r = \Mod \psi^{-1}(\widetilde{\Gamma})$ (see \cite[Section 7.5 and Section 11.3]{Vais1}), and hence $\Mod \widetilde{\Gamma}_r = \Mod \widetilde{\Gamma}$. 

We make the following crucial observation. For any curve $\widetilde{\gamma}: [0,1] \rightarrow \mathbb{R}^2$ contained in $\widetilde{\Gamma}_r$ such that there exists $t \in (0,1)$ with $\widetilde{\gamma}(t) \in \widetilde{Y}$ an accumulation point of $\im \widetilde{\gamma} \cap \widetilde{Y}$, it follows that $\im \widetilde{\gamma} \subset \widetilde{Y}$. This is a consequence of the identity theorem for analytic mappings, since $\widetilde{\gamma}$ is the image of a radial line segment under the analytic map $\psi$. There are at most countably many curves $\widetilde{\gamma} \in \widetilde{\Gamma}_r$ with $\im \widetilde{\gamma} \subset \widetilde{Y}$. By removing any such curves, and also removing any unrectifiable curves, we see there is a family of rectifiable curves $\widetilde{\Gamma}_r^0 \subset \widetilde{\Gamma}_r$ such that $\Mod \widetilde{\Gamma}_r^0 = \Mod \widetilde{\Gamma}_r$ and such that for any curve $\widetilde{\gamma}: [0,1] \rightarrow \mathbb{R}^2$ in $\widetilde{\Gamma}_r^0$ and any $t \in (0,1)$, $\widetilde{\gamma}(t)$ is not an accumulation point of $\im \gamma \cap \widetilde{Y}$. 

Now let $\rho$ be admissible for $\Gamma$, where we may assume that $\rho|_Y = 0$ by Proposition \ref{prop:hausdorff_intersection}. We may also assume that $\int_{\mathbb{G}_\alpha^2} \rho^2 d\mathcal{H}_\alpha^2 < \infty$. We claim that $\widetilde{\rho}(u,v) = |x(u)|^{-\alpha} \rho(x(u),v)\chi_{\mathbb{R}^2 \setminus \widetilde{Y}}$ is admissible for $\widetilde{\Gamma}_r^0$. We must show that $\int_{\gamma} \widetilde{\rho}\, ds_E \geq 1$ for all $\gamma \in \widetilde{\Gamma}_r^0$. Let $\widetilde{\gamma}: [0,1] \rightarrow \mathbb{R}^2$ be a curve in $\widetilde{\Gamma}_r^0$. As usual, let $\gamma = \varphi^{-1} \circ \widetilde{\gamma}$ be the corresponding curve in $\Gamma$. Our main obstacle is that $\gamma$ may not be a rectifiable curve. 

To work around this, let $\epsilon>0$ and construct a new curve $\widetilde{\gamma}_\epsilon$ by redefining $\widetilde{\gamma}$ on certain intervals as follows. Consider first the case that $\widetilde{\gamma}(0) \in \widetilde{Y}$. Since $\int_{\mathbb{R}^2} \widetilde{\rho}^2 dm = \int_{\mathbb{G}_\alpha^2} \rho^2 d\mathcal{H}_\alpha^2< \infty$, we may apply Lemma \ref{lemm:path_redefine} to find a value $r>0$ sufficiently small so that the circle $S(\widetilde{\gamma}(0),r)$ intersects $E$ and a point $\widetilde{\gamma}(t_\epsilon) \notin \widetilde{Y}$ for some $t_\epsilon \in (0,1)$, and such that $\int_{S(\widetilde{\gamma}(0),r)} \widetilde{\rho}\, ds_E < \epsilon$. Define $\widetilde{\gamma}_\epsilon$ on the interval $[0, t_\epsilon]$ by traversing that portion of the circle from $\widetilde{\gamma}(0)$ to $\widetilde{\gamma}(t_\epsilon)$. Notice that $\int_{\widetilde{\gamma}_\epsilon|_{[0, t_\epsilon]}} \widetilde{\rho}\, ds_E < \epsilon$. Next, in the case that $\widetilde{\gamma}(1) \in \widetilde{Y}$, we will define $\widetilde{\gamma}_\epsilon$ on some interval $[t_\epsilon',1]$ in a similar way, so that $\widetilde{\gamma}_\epsilon(t_\epsilon') = \widetilde{\gamma}(t_\epsilon') \notin \widetilde{Y}$, $\widetilde{\gamma}_\epsilon(1) \in F$, and $\int_{\widetilde{\gamma}_\epsilon|_{[t_\epsilon',1]}} \widetilde{\rho}\, ds_E < \epsilon$. Notice that there are finitely many values $t_1, t_2, \ldots, t_n \in [t_\epsilon, t_{\epsilon'}]$ for which $\gamma(t_j) \in \widetilde{Y}$. Define $\widetilde{\gamma}_\epsilon$ in a similar manner on a small interval $I_j$ around each $t_j$ so that $\sum_{j=1}^n \int_{\widetilde{\gamma}_\epsilon|_{I_j}} \widetilde{\rho}\, ds_E < \epsilon$. Define $\widetilde{\gamma}_\epsilon$ for all other values $t$ by $\widetilde{\gamma}_\epsilon(t) = \widetilde{\gamma}(t)$. Let $J_\epsilon = \{t \in [0,1]: \widetilde{\gamma}_\epsilon(t) = \widetilde{\gamma}(t)\}$; observe that $\int_{\widetilde{\gamma}_\epsilon|_{[0,1] \setminus J_\epsilon}}\widetilde{\rho}\,ds_E < 3\epsilon$. We have illustrated the construction of the curve $\widetilde{\gamma}_\epsilon$ in Figure \ref{fig:redefined_curve}.

\begin{figure} \centering
\includegraphics[height=2.25in]{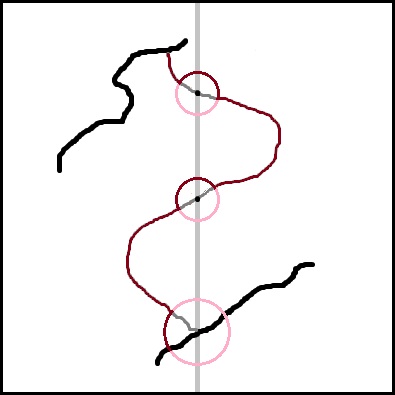}
\caption{The curve $\widetilde{\gamma}_\epsilon$ in the proof of Proposition \ref{prop:varphi_gqc} } \label{fig:redefined_curve}
\end{figure}

Observe that $\tilde{\gamma}_\epsilon$ only intersects the singular line along circles whose centers are on the singular line. Letting $\gamma_\epsilon = \varphi^{-1}\circ \widetilde{\gamma}_\epsilon$, it is easy to check that $\gamma_\epsilon$ is rectifiable and hence that $\int_{\gamma_\epsilon} \rho\, ds_E \geq 1$. By the change of variables calculation, it follows that $\int_{\widetilde{\gamma}_\epsilon} \widetilde{\rho}\, ds_E \geq 1$. This in turn implies that $\int_{\widetilde{\gamma}} \widetilde{\rho}\, ds_E \geq 1-3\epsilon$. Letting $\epsilon \rightarrow 0$, we conclude that $\int_{\widetilde{\gamma}} \widetilde{\rho}\, ds_E \geq 1$ for all $\widetilde{\gamma} \in \widetilde{\Gamma}_r^0$ and hence that $\widetilde{\rho}$ is admissible for $\widetilde{\Gamma}_r^0$. Thus $\Mod \widetilde{\Gamma}_r^0 = \Mod \widetilde{\Gamma} \leq \Mod \Gamma$, and the proof is complete. 
\end{proof}

We point out the following corollary to the preceding proposition. 

\begin{cor}
The $\alpha$-Grushin plane $\mathbb{G}_\alpha^2$ is a Loewner space with exponent 2. That is, for each $t>0$, there is a uniform positive lower bound on $\Mod \Gamma(E,F)$ (the family of curves connecting the disjoint continua $E$ and $F$) for all $E,F$ satisfying $$\frac{\dist(E,F)}{\min\{\diam(E),\diam(F)\}} \leq t.$$
\end{cor} 

This follows from the Loewner property for Euclidean space, since the Loewner property is preserved under a quasisymmetric map which is geometrically conformal.

The implication (a) $\Longrightarrow$ (d) is now immediate.

\begin{prop}\label{prop:d_implies_a}
Let $f: \mathbb{G}_\alpha^2 \rightarrow \mathbb{G}_\alpha^2$ be a quasisymmetric homeomorphism. Then $f$ is geometrically quasiconformal.
\end{prop}
\begin{proof}
The conjugated map $\widetilde{f} = \varphi \circ f \circ \varphi^{-1}: \mathbb{R}^2 \rightarrow \mathbb{R}^2$ is quasisymmetric and hence geometrically quasiconformal. Hence $f = \varphi^{-1} \circ \widetilde{f} \circ \varphi$ is the composition of geometrically quasiconformal maps and thus geometrically quasiconformal. 
\end{proof}

\begin{rem}\label{rem:quantitative}
Theorem \ref{thm:main_theorem} can be made quantitative by following the above proofs and using the quantitative nature of the equivalence of definitions of Euclidean quasiconformal mappings. 
\begin{itemize}
\item[(1)] As is true in general for quasisymmetric maps between metric spaces, if $f$ is $\eta$-quasisymmetric, then $f$ is $H$-metrically quasiconformal with $H = \eta(1)$.
\item[(2)] If $f$ is $H$-metrically quasiconformal on its Riemannian set, then $f$ is locally $K$-geometrically quasiconformal on its Riemannian set, in fact with $K=H$. This can be seen from \cite[Remark 34.2]{Vais1} and the fact that $\varphi$ is geometrically conformal. In fact, $f$ is also $K$-geometrically quasiconformal (without restriction to its Riemannian set) since the constant $K$ is not increased when applying \cite[Theorem 35.1]{Vais1} as we did when proving Proposition \ref{thm:qs_equivalence}.
\item[(3)] If $f$ is $K$-geometrically quasiconformal, then $f$ is $\eta$-quasisymmetric with $\eta$ determined as follows. First, 
it is known that a planar homeomorphism $\widetilde{f}: \mathbb{R}^2 \rightarrow \mathbb{R}^2$ which is $K$-geometrically quasiconformal is $\eta_0$-quasisymmetric for $\eta_0(t) = C(K)\max\{t^{K},t^{1/K}\}$, 
where we may take $C(K) = 4^{K-1} \exp(6(K+1)^2\sqrt{K-1})$. This bound is due to Vuorinen \cite[Theorem 1.8]{Vuo:90}; see also earlier work of Lehto, Virtanen, and V\"ais\"al\"a \cite{LVV:59}.  
Next, let $\eta_1$ denote the control function for the quasisymmetric map $\varphi$ and $\eta_1'$ denote the control function for the quasisymmetric map $\varphi^{-1}$. Then $f$ is $\eta$-quasisymmetric for $\eta = \eta_1'\circ \eta_0 \circ \eta_1$.
\item[(4)] If $f$ is $K$-geometrically quasiconformal, then $f$ satisfies the analytic definition in (f) with $\|\mu\|_\infty = (K-1)/(K+1)$. The converse also holds. This follows from a standard fact for Euclidean quasiconformal mappings, applied to the conjugated map $\widetilde{f}: \mathbb{R}^2 \rightarrow \mathbb{R}^2$.  
\end{itemize}
\end{rem} 

\section{Characterization of conformal maps} \label{sec:conformal_maps}

Ackermann \cite{Ack} raises the question of characterizing all orientation-preserving homeomorphisms of the Grushin plane which are metrically conformal, that is, metrically quasiconformal with $H=1$. She considers three classes of maps (originating in work of Payne \cite{Payne}), which she shows to be conformal: dilations, vertical translations, and rotation by a half revolution. We can verify that in fact all conformal homeomorphisms of the Grushin plane must be the composition of maps of these types. This a rigidity theorem similar to (and deriving from) the well-known classification of conformal homeomorphisms of the Euclidean plane. A similar rigidity result, using more sophisticated arguments, is found in Morbidelli \cite{Mor:09} for maps between domains in certain higher-dimensional Grushin spaces, 
though the case of the Grushin plane is not explicitly addressed there. 

\begin{thm} \label{thm:conformal_maps}
Let $f: \mathbb{G}_\alpha^2 \rightarrow \mathbb{G}_\alpha^2$ $(\alpha>0)$ be an orientation-preserving homeomorphism which is metrically conformal. Then there exists $\lambda>0, a\in \mathbb{R}$ 
such that $f(x,y) = (\pm \lambda x,\pm \lambda^{1+\alpha}y + a)$. 
\end{thm}
\begin{proof}
Define $\widetilde{f}: \mathbb{R}^2 \rightarrow \mathbb{R}^2$ by $\widetilde{f} = \varphi \circ f \circ \varphi^{-1}$. Note that $\widetilde{f}$ is a conformal mapping on its Riemannian set. By the argument of Proposition \ref{thm:qs_equivalence}, we conclude that $\widetilde{f}$ is in fact conformal on all of $\mathbb{R}^2$. It follows that 
$$\widetilde{f}(u,v) = (a_1 + \lambda_0((\cos \theta) u - (\sin \theta)v), a_2 + \lambda_0((\sin \theta) u + (\cos \theta)v))$$ 
for some $a_1,a_2 \in \mathbb{R}, \lambda_0>0, \theta \in [0, 2\pi)$. 
The theorem will follow by showing $\theta \in \{0, \pi\}$ and $a_1 = 0$. This in turn will follow from showing that $f$ maps the singular line $Y = \{0\} \times \mathbb{R}$ onto itself. 

An explicit formula for $f$ obtained by conjugating $\widetilde{f}$ by $\varphi$ is fairly complicated. However, we can work with $\widetilde{f}$ directly by considering instead the ``conformal'' definition of the Grushin plane rather than the standard definition. This is obtained by pushing the Grushin line element forward under the canonical quasisymmetry $\varphi$; see \cite[Definition 3.1]{Rom} for details. In the conformal Grushin plane, all balls centered on the singular line are self-similar (relative to the Euclidean metric) by the corresponding version of the dilation property (\ref{equ:grushin_dilation}) and have (Euclidean) dilatation strictly larger than 1. 
Since $\widetilde{f}$ is a similarity map, it preserves the Euclidean dilatation of metric balls; this implies that $\widetilde{f}$ cannot map a point on the singular line to a point off the singular line, as small balls in the conformal Grushin metric which are off the singular line have roughly Euclidean shape. We conclude that $\widetilde{f}$ and hence $f$ maps the singular line onto itself, and the result follows. 
\end{proof}

The situation here becomes somewhat more complicated if one considers conformal homeomorphisms $f: D \rightarrow D'$ between two domains $D, D' \subset \mathbb{G}_\alpha^2$, since we no longer have the same simple description of (Euclidean) conformal maps between $\varphi(D)$ and $\varphi(D')$. However, it is easy to see that the same basic conclusion that $f$ must map $Y \cap D$ into $Y$ still holds, since locally a (Euclidean) conformal map is close to a similarity map. Conversely, for any conformal map $\widetilde{f}: \varphi(D) \rightarrow \varphi(D')$ which maps $\widetilde{Y} \cap \varphi(D)$ into $\widetilde{Y}$, the corresponding map $f: D \rightarrow D'$ in the Grushin plane is still conformal. 

In a similar vein we can record the following fact.

\begin{prop}
There exists $H_0 = H_0(\alpha) >1$ such that every metrically quasiconformal homeomorphism $f: \mathbb{G}_\alpha^2 \rightarrow \mathbb{R}^2$ or $f: \mathbb{R}^2 \rightarrow \mathbb{G}_\alpha^2$ satisfies $H \geq H_0$. In particular, no such homeomorphism exists which is metrically conformal. 
\end{prop}
\begin{proof}
For simplicity, we consider the case of a mapping $f: \mathbb{R}^2 \rightarrow \mathbb{G}_\alpha^2$. Assume that $f$ is metrically quasiconformal with constant $H$. The composed map $\varphi \circ f: \mathbb{R}^2 \rightarrow \mathbb{R}^2$ is then metrically quasiconformal on the set $\mathbb{R}^2 \setminus Y$ with the same constant $H$. Hence $\varphi \circ f$ is geometrically quasiconformal (with $K = H$) on all of $\mathbb{R}^2$ by the argument in Proposition \ref{thm:qs_equivalence}. From this it follows that $\varphi \circ f$ is $\eta$-quasisymmetric, with $\eta$ depending on $H$. Observe from Remark \ref{rem:quantitative} (3) that we may choose $\eta$ so that $\eta(1) \rightarrow 1$ as $H \rightarrow 1$. As in the proof of Theorem \ref{thm:conformal_maps}, this leads to a contradiction when $\eta(1)$ is sufficiently close to 1. The existence of such a constant $H_0(\alpha)$ follows.
\end{proof}

\section{A family of curves}\label{sec:exm}

We conclude the paper with an illustrative example. It is well-known that the $n$-modulus of the family of all nonrectifiable curves in $\mathbb{R}^n$ is zero. In contrast, we have the following result in the Grushin plane. 

\begin{thm}\label{thm:nonrectifiable_family}
For $\alpha \geq 1$, there exists a family $\Gamma$ of nonrectifiable curves in $\mathbb{G}_\alpha^2$ with $\Mod \Gamma > 0$. 
\end{thm}
\begin{proof}
We continue with the notation in Section \ref{sec:gqc}: we use $(x,y)$ for coordinates in $\mathbb{G}_\alpha^2$ and $(u,v)$ for coordinates in $\mathbb{R}^2$, and the use of $\text{ }\widetilde{ }\text{ }$ will indicate objects in $\mathbb{R}^2$.  

Consider the family $\widetilde{\Gamma}$ of curves $\widetilde{\gamma}_a(t) = (t^{1+\alpha}/(1+\alpha), -t^\alpha/\log t + a): (0, 1/2] \rightarrow (\mathbb{R}^2,d_E)$, where $0 \leq a \leq 1$. We will show that this family has positive modulus. This has the same modulus as the family $\Gamma$ consisting of the curves $\gamma_a = \varphi^{-1} \circ \widetilde{\gamma}_a$. 

Let $D$ be the domain in $\mathbb{R}^2$ which is foliated by the curves in $\widetilde{\Gamma}$. In coordinates
$$\left\{ \begin{array}{rl} u & = \frac{t^{1+\alpha}}{1+\alpha} \\ v & = \frac{-t^\alpha}{\log t} + a \end{array} \right. \hspace{.5in} \left\{ \begin{array}{rl} t & = ((1+\alpha)u)^{1/(1+\alpha)} \\ a & = v + \frac{((1+\alpha)u)^{\alpha/(1+\alpha)}}{\log ((1+\alpha)u)^{1/(1+\alpha)}} \end{array} \right. . $$
Write $\widetilde{\gamma}(a,t)$ in place of $\widetilde{\gamma}_a(t)$, observing that $|J _{\widetilde{\gamma}}(a,t)| = t^\alpha$ and $|J_{ \widetilde{\gamma}^{-1}}(u,v)| = ((1+\alpha)u)^{-\alpha/(1+\alpha)}$. 

Let $\widetilde{\rho}$ be an admissible function for $\widetilde{\Gamma}$, which we may assume to be supported on $D$. This means that 
$\int_{\widetilde{\gamma}_a} \widetilde{\rho}\, ds \geq 1$
for all $a \in [0,1]$. This can be written as
\begin{align*}
1 \leq \int_{\widetilde{\gamma}_a} \widetilde{\rho}\, ds_E & = \int_0^{1/2} (\widetilde{\rho} \circ \widetilde{\gamma}_a) \sqrt{\left(\frac{du}{dt}\right)^2 + \left(\frac{dv}{dt}\right)^2} dt \\
& = \int_0^{1/2} (\widetilde{\rho} \circ \widetilde{\gamma}_a) \sqrt{t^{2\alpha} + \left(\frac{t^{-1+\alpha}}{\log^2 t} - \frac{\alpha\, t^{-1+\alpha}}{\log t} \right)^2} dt.
\end{align*}
Then 
\begin{align*}
1 & \leq  \int_0^1 \int_0^{1/2}  (\widetilde{\rho} \circ \widetilde{\gamma}_a) \sqrt{t^{2\alpha} + \left(\frac{t^{-1+\alpha}}{\log^2 t} - \frac{\alpha\, t^{-1+\alpha}}{\log t}\right)^2}\, dt\,da \\
 & = \int_D \widetilde{\rho} \sqrt{((1+\alpha)u)^{2\alpha/(1+\alpha)} + \left(\frac{t(u)^{-1+\alpha}}{\log^2 t(u)} - \frac{\alpha\, t(u)^{-1+\alpha}}{\log t(u)}\right)^2} \frac{1}{((1+\alpha)u)^{\alpha/(1+\alpha)}}\, dm . \\
\end{align*}
Applying H\"older's inequality gives
\begin{align*}
1 \leq \left(\int_D \widetilde{\rho}^2dm\right)^{1/2} \left( \int_D 1 + \frac{1}{((1+\alpha)u)^{2\alpha/(1+\alpha)}}\left(\frac{t(u)^{-1+\alpha}}{\log^2 t(u)} - \frac{\alpha\, t(u)^{-1+\alpha}}{\log t(u)}\right)^2\, dm \right)^{1/2}.
\end{align*}
To evaluate the rightmost integral, we pull back to the $(a,t)$-plane to get
\begin{align*}
\int_D \left(1 + \frac{1}{t^{2\alpha}}\left(\frac{t^{-1+\alpha}}{\log^2 t} - \frac{\alpha\, t^{-1+\alpha}}{\log t} \right)^2 \right) t^\alpha\, dm & = \int_0^1\int_0^{1/2} t^\alpha + \frac{1}{t^\alpha}\left(\frac{t^{-1+\alpha}}{\log^2 t} - \frac{\alpha\, t^{-1+\alpha}}{\log t} \right)^2\, dt\,da \\
& = \int_0^{1/2} t^\alpha + \frac{1}{t^\alpha}\left(\frac{t^{-1+\alpha}}{\log^2 t} - \frac{\alpha\, t^{-1+\alpha}}{\log t} \right)^2\, dt . 
\end{align*}
The convergence of this integral follows from that of $\int_0^{1/2} -t^{-2+\alpha} (\log t)^{-2}\,dt$, for $\alpha \geq 1$. This gives a positive lower bound for $\int_D \widetilde{\rho}^2\,  dm$ in the case that $\alpha \geq 1$.

However, we show that the curves $\gamma_a(t) = (t, \frac{-t^\alpha}{\log t} + a)$ are not rectifiable for $\alpha > 0$. We compute
\begin{align*} 
\int_{\gamma_a} ds_\alpha & = \int_0^{1/2} \sqrt{1 + \frac{1}{t^{2\alpha}}\left(\frac{t^{-1+\alpha}}{\log^2 t} - \frac{\alpha\, t^{-1+\alpha}}{\log t}\right)^2}\, dt \\ & \geq \int_0^{1/2} \frac{|1-\alpha \log t|}{t \log^2 t} \, dt = \infty,
\end{align*}
where divergence follows from the divergence of $\int_0^{1/2} (-t \log t)^{-1}\,dt$, for $\alpha > 0$. 
\end{proof}

Based on the same example we also obtain the following result, which is essentially a corollary of the previous proof. This highlights some further differences with the case of Euclidean quasiconformal mappings and underscores the choice of geometric definition used in Theorem \ref{thm:main_theorem}.

\begin{cor} \label{cor:failure}
Let $f: \mathbb{G}_\alpha^2 \rightarrow \mathbb{G}_\alpha^2$ ($\alpha \geq 1$) be a quasiconformal homeomorphism (in any of the senses of Theorem \ref{thm:main_theorem}). 
\begin{itemize}
\item[(1)] The modulus of a curve family may fail to be quasi-preserved by $f$. Of necessity, such a curve family does not correspond to a ring domain.  
\item[(2)] (failure of absolute continuity on lines (ACL) property) The map $f$ may fail to be absolutely continuous on almost every closed, compact curve. 
\item[(3)] (failure of Lusin condition $N$) The map $f$ may map a set of zero measure (here, Hausdorff 2-measure) onto a set of positive measure. 
\end{itemize}
These statements also hold in the case of a quasiconformal homeomorphism $f: \mathbb{R}^2 \rightarrow \mathbb{G}_\alpha^2$. 
\end{cor}
\begin{proof}
Let $\Gamma'$ be the family of curves in $\mathbb{G}_\alpha^2$ obtained by adding the left endpoint 0 into the domain of each curve in $\Gamma$ in the example of Theorem \ref{thm:nonrectifiable_family}. Then $\Gamma'$ consists only of non-locally rectifiable curves. The definition of modulus implies that any curve family with no non-locally rectifiable curves has modulus zero. However, the family $\varphi \Gamma'$ has the same modulus as the family $\widetilde{\Gamma}$ in the proof of Theorem \ref{thm:nonrectifiable_family}, which is positive. Hence it is not true that the map $\varphi$ quasi-preserves the modulus of all curve families not associated to some disjoint continua $E,F$. 

Define $\widetilde{f}: \mathbb{R}^2 \rightarrow \mathbb{R}^2$ by $\widetilde{f}(u,v) = ( u+1,v)$, and $f: \mathbb{G}_\alpha^2 \rightarrow \mathbb{G}_\alpha^2$ by $f = \varphi^{-1} \circ \widetilde{f} \circ \varphi$. Since $\varphi \Gamma'$ and $f\Gamma'$ have the same modulus, we obtain claim (1) of the theorem. 

To verify claim (2), consider again the curve family $\Gamma'$ and the function $f$ defined in the previous paragraph. The family $f \Gamma'$ has positive modulus and contains only rectifiable curves, but every curve in $\Gamma'$ is nonrectifiable. This shows that $f^{-1}$ is not absolutely continuous on almost every closed rectifiable curve. 

To verify claim (3) observe that $f(Y)$ is a locally rectifable line in $\mathbb{G}_\alpha^2$, whose image under $f^{-1}$, namely $Y$ itself, has positive 2-measure whenever $\alpha \geq 1$.

For the final statement, consider the mapping $\varphi^{-1} \circ \widetilde{f}: \mathbb{R}^2 \rightarrow \mathbb{G}_\alpha^2$. In contrast, observe that claims (2) and (3) are false for the class of quasiconformal homeomorphisms from $\mathbb{G}_\alpha^2$ to $\mathbb{R}^2$. 
\end{proof}

\noindent {\bf Acknowledgments.} D. Jung was supported by U.S. Department of Education GAANN fellowship P200A150319. The authors thank Jeremy Tyson and Colleen Ackermann for comments on a draft of this paper. They also thank the referee for useful feedback.

\bibliographystyle{abbrv}
\bibliography{biblio}

\def\cprime{$'$} \def\cprime{$'$}
\begin{thebibliography}{10}

\bibitem{Ack}
C.~Ackermann.
\newblock An approach to studying quasiconformal mappings on generalized
  {G}rushin planes.
\newblock {\em Ann. Acad. Sci. Fenn.}, 40(1):305--320, 2015.

\bibitem{Bellaiche}
A.~Bella{\"{\i}}che.
\newblock The tangent space in sub-{R}iemannian geometry.
\newblock In {\em Sub-{R}iemannian geometry}, volume 144 of {\em Progr. Math.},
  pages 1--78. Birkh\"auser, Basel, 1996.

\bibitem{hei:lectures}
J.~Heinonen.
\newblock {\em Lectures on analysis on metric spaces}.
\newblock Universitext. Springer-Verlag, New York, 2001.

\bibitem{HKST:01}
J.~Heinonen, P.~Koskela, N.~Shanmugalingam, and J.~T. Tyson.
\newblock Sobolev classes of {B}anach space-valued functions and quasiconformal
  mappings.
\newblock {\em J. Anal. Math.}, 85:87--139, 2001.

\bibitem{KorRei:95}
A.~Kor{\'a}nyi and H.~M. Reimann.
\newblock Foundations for the theory of quasiconformal mappings on the
  {H}eisenberg group.
\newblock {\em Adv. Math.}, 111(1):1--87, 1995.

\bibitem{LVV:59}
O.~Lehto, K.~I. Virtanen, and J.~V\"ais\"al\"a.
\newblock Contributions to the distortion theory of quasiconformal mappings.
\newblock {\em Ann. Acad. Sci. Fenn. Ser. A I No.}, 273:14, 1959.

\bibitem{MarMos:95}
G.~A. Margulis and G.~D. Mostow.
\newblock The differential of a quasi-conformal mapping of a
  {C}arnot-{C}arath\'eodory space.
\newblock {\em Geom. Funct. Anal.}, 5(2):402--433, 1995.

\bibitem{Meyerson}
W.~Meyerson.
\newblock The {G}rushin plane and quasiconformal {J}acobians.
\newblock {\em arXiv preprint arXiv:1112.0078}, 2011.

\bibitem{MonMor1}
R.~Monti and D.~Morbidelli.
\newblock Isoperimetric inequality in the {G}rushin plane.
\newblock {\em J. Geom. Anal.}, 14(2):355--368, 2004.

\bibitem{Mor:09}
D.~Morbidelli.
\newblock Liouville theorem, conformally invariant cones and umbilical surfaces
  for {G}rushin-type metrics.
\newblock {\em Israel J. Math.}, 173:379--402, 2009.

\bibitem{Payne}
K.~R. Payne.
\newblock Singular metrics and associated conformal groups underlying
  differential operators of mixed and degenerate types.
\newblock {\em Ann. Mat. Pura Appl. (4)}, 185(4):613--625, 2006.

\bibitem{Raj16}
K.~Rajala.
\newblock Uniformization of two-dimensional metric surfaces.
\newblock {\em Invent. Math.}, to appear.

\bibitem{Ric:93}
S.~Rickman.
\newblock {\em Quasiregular mappings}, volume~26 of {\em Ergebnisse der
  Mathematik und ihrer Grenzgebiete (3) [Results in Mathematics and Related
  Areas (3)]}.
\newblock Springer-Verlag, Berlin, 1993.

\bibitem{Rom}
M.~Romney.
\newblock Conformal {G}rushin spaces.
\newblock {\em Conform. Geom. Dyn.}, 20:97--115, 2016.

\bibitem{RV}
M.~Romney and V.~Vellis.
\newblock Bi-{L}ipschitz embedding of the generalized {G}rushin plane in
  {E}uclidean spaces.
\newblock {\em Math. Res. Lett.}, to appear.

\bibitem{Vais1}
J.~V{\"a}is{\"a}l{\"a}.
\newblock {\em Lectures on {$n$}-dimensional quasiconformal mappings}.
\newblock Lecture Notes in Mathematics, Vol. 229. Springer-Verlag, Berlin-New
  York, 1971.

\bibitem{Vuo:90}
M.~Vuorinen.
\newblock Quadruples and spatial quasiconformal mappings.
\newblock {\em Math. Z.}, 205:617--628, 1990.

\bibitem{Wu:grushin}
J.-M. Wu.
\newblock Geometry of {G}rushin spaces.
\newblock {\em Illinois J. Math.}, 59(1):21--41, 2015.

\end{thebibliography}

\end{document}